\documentclass{amsart}

\usepackage{amssymb}

\usepackage{graphicx}
\usepackage{tikz}
\usetikzlibrary{decorations.markings}
\usepackage{istgame}
\usepackage{}
\usepackage{cellspace}%
\usepackage{multirow}
\usepackage{mathtools}
\usepackage{rotating}
\usepackage{mathrsfs}

\newtheorem{theorem}{Theorem}[section]
\newtheorem{lemma}[theorem]{Lemma}
\newtheorem{corollary}[theorem]{Corollary}
\newtheorem{prop}[theorem]{Proposition} 

\theoremstyle{definition}

\newtheorem{example}[theorem]{Example}

\theoremstyle{remark}

\numberwithin{equation}{section}
\newcommand{\bigo}[1]{O\left( {#1} \right)}
\begin{document}



\thanks{}

\author{}
\address{}
\curraddr{}
\email{}
\thanks{}

\title{Lengths of Paths In Rooted Trees}
\author{Keith Copenhaver}
\email{keithc@ufl.edu}
\date{}


\begin{abstract}
We provide formulas for generating functions of many types of paths in various rooted tree structures. We compute the $k$th moment of the generating functions for various types of vertical paths. In two specific familes of trees we find exact closed formulas for expectations and their asymptotic values. Some of these closed formulas are surprisingly simple.


\end{abstract}

\maketitle
\section{Introduction}

A tree is a connected graph without cycles. A tree is rooted if it has one vertex designated as the root. In a rooted tree there is a natural orientation imposed on the vertices, the neighbor of a vertex which is closer to the root is called the parent, and the vertices which are further away are called the children. Extending the metaphor, if a vertex $v$ is contained in the subtree rooted at a vertex $u$, then $v$ is a descendant of $u$, and $u$ is an ancestor of $v$.\\

Trees have widespread applications. They are used to model networks of many kinds, and they are used as data structures. In some of these applications, leaves (vertices which have no children) have particular significance. In data structures, the leaves will always be the last thing called in a search algorithm, so one might like the leaves to be relatively close to the root. If we represent a computer network as a rooted tree, then the leaves are the nodes with the least amount of access, and presumably, the least security; in this application, we might want most vertices to be far from leaves. Height is the length of the longest path from the root; this path necessarily ends at a leaf. Height has been a longstanding topic of study. There have been a number of papers recently about protection and rank, starting with the seminal paper by Cheon and Shapiro \cite{cheon}. A vertex is $k$-protected if its closest leaf descendant is at least $k$ steps away, and it has rank $k$ if its closest leaf descendant is exactly $k$ steps away. Height and rank provide two halves of the same coin: on some level, they both address the question ``how far away are the leaves from the root?'' It is a decades old result from de Bruijn, Knuth, and Rice that the \textit{furthest} leaf in a uniformly selected general tree of arbitrarily large size is expected to be arbitrarily far away \cite{de_bruijn_knuth_rice_1972}. This was later proven for binary trees by Flajolet and Odlyzko as well\cite{flajolet_odlyzko_1982}. It is a recent result that the \textit{closest} leaf in general trees is expected to be less than two edges away \cite{copenhaver_2017}. This leaves the question of what exactly happens in the middle? We will answer that question precisely with both explicit formulas and asymptotic estimations. \\

This also leads to generalization: what if one were to pick two vertices arbitrarily? What if we pick a vertex and any of its descendants? We will call paths between a vertex and one of its descendants ``vertical paths.''\\

We calculate the $k$th moments of path lengths of certain vertical paths in certain simple varieties of trees, and also consider general paths in two specific families of trees. General trees are trees in which every vertex may have any number of children. Binary trees are trees in which every vertex may have zero, one, or two children, where single children may be left or right children. \\
There are some interesting relationships in these two families, for example:

\begin{itemize}
\item The expected length of a path from the root to a leaf is very close to the average of the expected distances to the closest and furthest leaves.
\item Even downward paths, the paths which have the shortest expected length, are expected to be arbitrarily long in arbitrarily large trees.
\item In these families of trees, insisting that paths end in leaves does not result in large difference in the expected length of downward paths; asymptotically the expectations differ by only a constant.
\item There is a great deal of overlap across the generating functions for both number of paths and number of edges in paths within each family. This situation is less pronounced in other families of trees, where the generating function for the family is less pleasant.
\item The expected length of a path from a root to a leaf is asymptotically equal to the expected length of an arbitrary path in trees on $n$ vertices, and, in the case of general trees, these expectations are exactly equal for all $n$.
\end{itemize}
\section{Vertical Paths}

In horticulture, grafting is a process where a branch of one tree is removed and made to grow as a branch of a different tree. We will use a lemma which does something similar with graph theoretic trees.

\begin{lemma}[Grafting Lemma]

Let $\mathcal T$ be a class of trees where the leaves are distinguishable and any subtree is a valid tree. Then there is a bijection between the set of vertices in trees of size $n$ in $\mathcal T$ and the set of ordered pairs of leaves in trees of size $n-k+1$ and trees of size $k$, with $1 \leq k \leq n$. Further, this bijection extends to allow marking vertices or edges and restricts to subsets of subtrees and trees.

\end{lemma}

\begin{proof}
Fix a vertex and remove the subtree rooted at that point, replacing it with a leaf. This gives a tree with a marked leaf and another tree. Conversely, fix a tree with a marked leaf and another tree, remove the leaf, and replace it by placing the tree at that vertex as a subtree. Note that if there were any marked edges or vertices in the tree, they will now correspond to marked vertices or edges in the designated subtree. 
\end{proof}
We will call a family of trees where the Grafting Lemma applies \textit{graftable}. On a basic level, the Grafting Lemma implies that 
\begin{equation}
V_\mathcal{T} = \frac{L_\mathcal{T}(x)}{x} T_{\mathcal T}(x),
\end{equation}
and, since 
$$\displaystyle V_\mathcal{T} (x) = \sum_{n=1}^\infty n \left( [x^n] T_\mathcal{T}(x) \right ) x^n=x T_\mathcal{T}^\prime(x),$$
rearranging yields 
$$L_\mathcal{T}(x) = \dfrac{x^2 T_\mathcal{T} ^\prime(x)}{T_\mathcal{T}(x)}.$$
This gives a method for computing the generating function for leaves in any case where the generating function for trees is known, without requiring bivariate generating functions or insights into the functional equation satisfied by the trees. However, if it is known that the generating function satisfies a functional equation of the form 
\begin{equation}
T_\mathcal{T}= x \Phi_\mathcal{T} (T_\mathcal{T}),
\end{equation}
 this relation implies that $x T_\mathcal{T} ^\prime (x)= T_\mathcal{T}(1- x \Phi_\mathcal{T} ^\prime (T_\mathcal{T}))^{-1},$ (this equation was noted in \cite{entringer_meir_moon_szekely_1994}) and comparing these equations yields

$$L_\mathcal{T}(x)= \dfrac{x}{1-x \Phi_\mathcal{T} ^\prime (T_\mathcal{T}(x))}.$$ \\ 

This bijection is powerful because it preserves many qualities of the tree (or root) being grafted. It is generally simpler to construct a generating function for trees with roots of a given type, and this immediately generalizes those results to arbitrary vertices. The core idea of this lemma was used in \cite{cheon}.\\

For a fixed family of trees $\mathcal T$, let the set $\{ T_\alpha \}$ be the set of trees in $\mathcal{T}$, let $r(T_\alpha)$ be the root of the tree $T_\alpha,$ let $d(v_1, v_2)$ be the number of edges in the path between vertices $v_1$ and $v_2$, and let $T (x),$ $V(x),$ and $L(x)$ be the generating functions for the number of trees, vertices, and leaves on $n$ vertices, respectively. \\

\begin{theorem}
Let $\mathcal{T}$ be a graftable family of trees satisfying (2.2). Let 
$$D_k(x) = \sum_{T_\alpha \in \mathcal T} \sum_{v \in T_\alpha} d(v, r(T_\alpha))^k x^{|T_\alpha|}.$$

Then, for $k \in \mathbb N,$ $D_k(x)=(V(x) - T(x))P_k \left(\frac{L(x)}{x} \right)$, where 
\begin{equation}
P_k \left(\frac{L(x)}{x} \right) = k! \left( \frac{L(x)}{x} \right)^k -  \frac{k!(k-1)}{2} \left( \frac{L(x)}{x} \right)^{k-1} + Q_k\left(\frac{L(x)}{x} \right),
\end{equation}
and $Q_k\left(\frac{L(x)}{x} \right)$ is a polynomial in $\frac{L(x)}{x}$ with of degree $k-2$ if $k \geq 2$ and $Q_k\left(\frac{L(x)}{x} \right)=0$ otherwise.

\end{theorem}

%

\begin{proof}
We proceed by induction on $k$. The statement for $k=0$ corresponds to simply counting the vertices which are not the root, which does indeed have generating function $V(x)-T(x)$. \\

If $k=1$, then $D_1(x)$ is the number of edges in all paths from the root, also known as the path length. Take any path from the root with a marked edge. Remove the subtree which begins at the top of the marked edge; this gives an ordered pair of a tree with a marked leaf which has been removed, enumerated by $\frac{L(x)}{x}$ and a tree with a path ending at a non-root vertex, enumerated by $V(x)-T(x)$. \\

In general, consider the generating function $ \left(\frac{L(x)}{x} \right)^k (V(x)-T(x))$. This corresponds to a sequence of trees with marked leaves which have been removed (allowing the singleton tree, with the single vertex removed), and a tree with a marked non-leaf vertex. Call the root of this tree $r$ and the marked vertex $v$. Graft each of the trees together, in order, attaching each at the location of the marked leaf. This gives a tree with a marked path terminating at $v$, and $k$ marks on the vertices which are ancestors of $v$. For each mark on a vertex, mark the edge below it. Since each edge may be marked multiple times, $ \left(\frac{L(x)}{x} \right)^k (V(x)-T(x))$ counts each path once for every $k$ element multiset of $l=d(v, r)$. See figure 1 for an illustration with $k=3.$

\begin{figure}[ht]
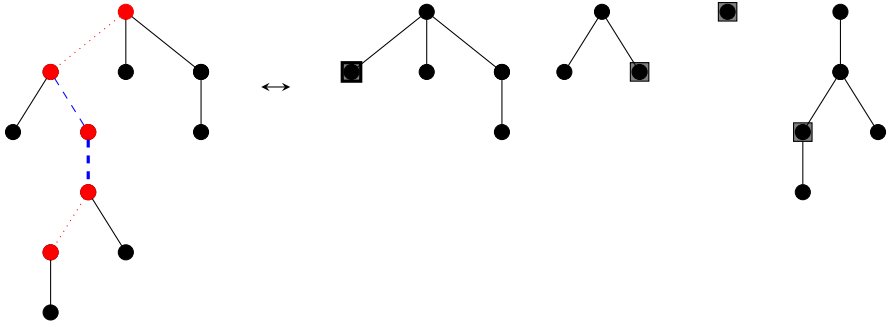

\begin{istgame}
\setistSolidNodeStyle{.2cm}
\setistRectangleNodeStyle{7pt}[black][.5]
\xtdistance{8mm}{10mm}
\xtShowEndPoints
\istroot(A01)(0,0)[red]
\istb[red, dotted]
\istb
\istb
\endist
\istroot(A11)(A01-1)[red]
\istb
\istb[blue, dashed]
\endist
\istroot(A13)(A01-3)
\istb
\endist
\istroot(A32)(A11-2)[red]
\istb[blue, dashed, very thick]
\endist
\istroot(A41)(A32-1)[red]
\istb[red, dotted]
\istb
\endist
\istroot(A51)(A41-1)[red]
\istb
\endist

\draw[black, <->,
    shorten >=0.4pt]
    (1.8,-1) -- (2.2,-1);

\istroot(B01)(4,0)
\istb
\istb
\istb
\endist
\istroot(B11)(B01-1)[rectangle node, black, very thick]
\endist
\istroot(B13)(B01-3)
\istb
\endist

\istroot(C01)(6.33,0)
\istb
\istb
\endist
\istroot(C12)(C01-2)[rectangle node, black]
\endist

\istroot(D01)(8,0)[rectangle node, black]
\endist

\istroot(D01A)(8,0)
\endist

\istroot(E01)(9.5,0)
\istb
\endist
\istroot(E11)(E01-1)
\istb
\istb
\endist
\istroot(E21)(E11-1)[rectangle node]
\istb
\endist
\end{istgame}
\caption{An illustration that $\left(\frac{L(x)}{x}\right)^3(V(x)-T(x))$ enumerates the number of 3 element multisets of the edges in a marked path. The second edge is marked once, and the third edge is marked twice, which corresponds to the triplet of trees with marked leaves and the tree with a marked vertex.}
\end{figure}

Thus
\begin{align*}
\left(\frac{L(x)}{x} \right)^k (V(x)-T(x)) &=  \sum_{T_\alpha \in \mathcal T} \sum_{v \in T_\alpha} {d(v, r(T_\alpha)) +k -1 \choose k} x^{|T_\alpha|}.
\end{align*}
Since
\begin{equation}
{n+k -1 \choose k}=\frac{1}{k!} \left(n^k - \frac{k(k-1)}{2}n^{k-1} + q_k(n) \right),
\end{equation}
where $q_k(n)$ is a polynomial in $n$ of degree $k-2$ if $k \geq 2$, and $q_1(n)=q_0(n)=0$.

Multiplying both sides by $k!$, substituting (2.4), and grouping powers of $d(v, r(T_\alpha))$, we have

$$\left(\frac{L(x)}{x} \right)^k (V(x)-T(x)) =  \frac{1}{k!}D_k(x) + \frac{k(k-1)}{2 k!} D_{k-1}(x) + ...,$$
where the ellipsis contains only rational expressions in $k$ multiplied by $D_i(x)$ with $0 \leq i \leq k-2.$ Solving for $D_k(x)$ and applying the induction hypothesis to the lower $D_i(x)$, we have

\begin{align*}
D_{k}(x) &= k! \left(\frac{L(x)}{x} \right)^k (V(x)-T(x)) - \frac{k(k-1)}{2} (k-1)! \left(\frac{L(x)}{x} \right)^{k-1} (V(x)-T(x)) \\
 &  + Q_k{\left(\frac{L(x)}{x} \right)(V(x)-T(x))}.
 \end{align*}
\end{proof}

\begin{corollary}
Let $P \left( \frac{L(x)}{x} \right)$ be the same as in (2.3). The generating functions for paths weighted by $k$th powers of their length with the following endpoints are as follows:
\begin{enumerate}
\item Paths between the root and leaves: $$\sum_{T_\alpha \in \mathcal T} \ \  \sum_{\mathclap{\substack{l \in T_\alpha \\ l \text{ is a leaf}}}} d(l, r(T_\alpha))^k x^{|T_\alpha|}=P \left( \frac{L(x)}{x} \right) (L(x)- [x^1] T(x)).$$
Note that $[x^1]T(x)$ indicates the coefficient of $x$ in the generating function, or the number of trees on a single vertex.
\item Paths between an arbitrary vertex and one of its descendants: $$\sum_{T_\alpha \in \mathcal T} \ \ \ \  \sum_{\mathclap{\substack{v_1, v_2 \in T_\alpha \\ v_2 \text{ is a descendant of }v_1}}}d(v_1, v_2)^k x^{|T_\alpha|}=P \left( \frac{L(x)}{x} \right) \frac{L(x)}{x}(V(x)-T(x)).$$
\item Paths between an arbitrary vertex and one of its descendant leaves: $$\sum_{T_\alpha \in \mathcal T} \ \ \sum_{\mathclap{\substack{v, l \in T_\alpha \\ l \text{ is a descendant of }v \\ l \text{ is a leaf}}}}d(v, v_2)^k=P \left( \frac{L(x)}{x} \right) \frac{L(x)}{x}(L(x)- [x^1] T(x)).$$
\end{enumerate}
\end{corollary}
\begin{proof}

\end{proof}
\begin{enumerate}
\item The argument is indentical to the proof of Theorem 2.2, but the final grafted tree should contain a marked leaf instead of a marked general vertex, so the role of $(V(x)-T(x))$ is replaced by the generating function for non-root leaves, which is $(L(x) - [x^1]T(x)).$
\item[(2), (3)] Both follow from applying the grafting lemma to Theorem 2.2 and (1).
\end{enumerate}

In a fixed family of trees $\mathcal T$, let $X(n)$ be a random variable whose value is the length of a path from the root of a tree to another vertex in that tree, where the path is selected uniformly from all such paths in trees on $n$ vertices. Similarly, let $X_l(n)$ be the length of a uniformly selected path from the root to a leaf, $X_{v, v}(n)$ be for vertical paths where there are no other restrictions on the endpoints, and let $X_{v, l}(n)$ be for vertical paths where there is no restriction on the upper vertex, but the lower must be a leaf.

\begin{theorem}
Let $\mathcal{T}$ be a graftable family of trees which satisfy (2.2) and with only one tree on a single vertex. Let $T(x)$ be the generating function for the number of trees on $n$ vertices, with radius of convergence $\rho$, and let $T(x)$ have expansion
 $$T(x)= a_0 - a_{1} \sqrt{1- \frac{x}{\rho}} + a_2 \left(1-\frac{x}{\rho} \right) +\bigo{\left(1- \frac{x}{\rho}\right)^{3/2}}$$
 
about $x = \rho$ (it is proven in \cite{flajoletsedgewick} that such an expansion exists, is computable using only $\phi$, and includes only terms of the form $(1-\rho x)^{k/2}$ for $k \in \mathbb Z, k \geq 0$).
Then the $k$th moment of the lengths of paths from the root is

\begin{align*}\mathbb{E}[X(n)^k]&=n^{k/2} \left(\left(\frac{a_1}{a_0}\right)^k \Gamma \left(\frac{k}{2}+1\right)- \right. \\
 &  \  \left. \frac{a_1^{k-1} \left((k+1) a_0^2+2 a_2 (k+1)a_0-a_1^2 k\right) k(k-1) \Gamma\left(\frac{k-1}{2} \right)}{4 a_0^{k+1}\sqrt{n}}\right)+\bigo{n^{k/2-1}},
\end{align*}
the $k$th moment of the length of paths from the root to leaves is
$$\mathbb{E}[X_l(n)^k]=\mathbb{E}[X^k] + n^{(k-1)/2} \left (\frac{a_1}{a_0} \right)^{k+1} \frac{k(k-1)}{4} \Gamma \left( \frac{k-1}{2} \right) + \bigo{n^{k/2-1}},$$
the $k$th moment of the length of arbitrary vertical paths is 
\begin{align*}
\mathbb{E}[X_v(n)^k] &=n^{k/2} \left(\frac{a_1^k \Gamma
   \left(\frac{k+1}{2}\right)}{a_o^k\sqrt{\pi }} \right. \\
& \ - \frac{a_1^{k-1}}{a_0^{k+1}\pi \sqrt{n}} \left(\left((k+1) a_0^2+2 a_2 (k+2) a_0-a_1^2 (k+1)\right) \sqrt{\pi } \Gamma
   \left(\frac{k}{2}+1\right) \right .\\
& \ \left. +\left(a_1^2-2 a_0 (a_0+2 a_2)\right) \Gamma
   \left(\frac{k+1}{2}\right) \vphantom{\frac{\left(\frac{a_1}{a_0}\right)^k \Gamma
   \left(\frac{k+1}{2}\right)}{\sqrt{\pi }}}\right)  + \bigo{n^{k/2-1}},\\ 
\end{align*}
and the $k$th moment of the length of vertical paths which end at a leaf is
\begin{align*}
\mathbb{E}[X_{v, l}(n)^k] &= \mathbb{E}[X_v(n)^k] + n^{(k-1)/2} \left(\frac{\left(\frac{a_1}{a_0}\right)^{k+1} \left(\sqrt{\pi } \Gamma
   \left(\frac{k}{2}+1\right)-\Gamma \left(\frac{k+1}{2}\right)\right) }{\pi
   }\right) \\
   & \ \ + \bigo{n^{k/2-1}}.
\end{align*}
\end{theorem}
\begin{proof} We present here the calculation for $D_k(x)$, the other calculations are very similar. Using the fact that $V(x)= x T ^\prime (x)$ and that $L(x)= \frac{x V(x)}{T(x)}$, this gives:
$$V(x) = \frac{a_1}{2 \sqrt{1-\frac{x}{\rho}}}-a_2+\bigo{\sqrt{1-\frac{x}{\rho}}},$$
$$L(x)=\frac{a_1 \rho}{2 a_0 \sqrt{1-\frac{x}{\rho}}}+\frac{\rho \left(a_1^2-2 a_0 a_2\right)}{2a_0^2}+\bigo{\sqrt{1-\frac{x}{\rho}}},$$
$$\frac{L(x)}{x} = \frac{a_1}{2 a_0 \sqrt{1- \frac{x}{\rho}}}+\frac{\left(a_1^2-2 a_0 a_2\right)}{2a_0^2}+\bigo{\sqrt{1- \frac{x}{\rho}}}.$$
Since

\begin{align}
\left(\frac{L(x)}{x} \right)^k  (V(x)-T(x)) = &\frac{a_1^{k+1}}{2^{k+1}a_0^k \left(1-\frac{x}{\rho }\right)^{\frac{1}{2} (k+1)}} \\
 \nonumber & \    +\frac{a_1^k  \left(a_1^2 k - 2a_0^2 -2 a_0 a_2 (k+1)\right)}{ 2^{k+1}a_0^{k+1}  \left(1-\frac{x}{\rho }\right)^{k/2} } \\
 \nonumber & \ + \bigo{\left(1-\frac{x}{\rho} \right)^{-k/2}}, 
\end{align}

plugging (2.5) into Theorem 2.2 yields that

\begin{align*}
D_k(x) &= \frac{a_1^{k+1}k!}{2^{k+1}a_0^k \left( 1 - \frac x \rho \right)^{\frac{k}{2}+1}} -\frac{   a_1^k k!\left( a_0^2(k+1)+2a_0 a_2 (k+1)-a_1^2 k\right)}{2^{k+1} a_0^{k+1} \left(1-\frac{x}{\rho}\right)^{\frac{k+1}{2} }}\\
\nonumber & \ + \bigo{\left(1-\frac{x}{\rho} \right)^{-k/2}}.
\end{align*}

Applying singularity analysis to $D_k(x)$ gives

\begin{align*}
[x^n]D_k(x) = & \rho^n \left( \frac{a_1^{k+1} k! }{2^{k+1} a_0^k\Gamma\left(\frac{k+1}{2}\right)}n^{\frac{k-1}{2}} -\frac{a_1^k  k! \left(a_0^2 (k+1)+2 a_0 a_2 (k+1)-a_1^2 k\right)}{ 2^{k+1} a_0^{k+1}\Gamma \left(\frac{k}{2}\right)}  n^{\frac{k}{2}-1} \right) \\
+& \bigo{n^{\frac{k-3}{2}}\rho^n}.
\end{align*}

\begin{align*}
[x^n]D_0(x)=\rho^n \left( \frac{a_1}{2 \sqrt{\pi}\sqrt{n}} -  \frac{1+2a_0 a_2}{2 a_0} \right)+ \bigo{\sqrt{n} \rho^n }
\end{align*}

Since $\mathbb{E}[X^k] = \frac{[x^n]D_k(x)}{[x^n]D_0(x)},$ the result follows. 
\end{proof}

\begin{example}
Let $\mathcal{C}$ be the family of Cayley Trees, labeled trees where the order of the children do not matter. In this case, since the generating functions are exponential, we need there to be a correspondence between an ordered triplet of a tree with a marked leaf, a tree, and a subset of the labels which corresponds to the labels used in the subtree. This is indeed still a bijection, so $\mathcal C$ is graftable. We have that $T_\mathcal{C}(x)= x e^{T_\mathcal{C}(x)},$ so $\mathcal C$ satisfies (2.2). This relation does not yield a closed form solution, but it is shown in \cite{flajoletsedgewick} that it has a singular expansion about its dominant singularity with $T_\mathcal{C}(x) = 1 - \sqrt{2}\sqrt{1- e x}+ \frac 2 3 (1-ex) + \bigo{(1-ex)^{3/2}}$. Thus, the $k$th moment of the lengths of vertical paths which end in leaves in $\mathcal{C}_n$ is
\begin{align*}
\mathbb{E}[X_{v, l}(n)^k] &= n^{k/2} \left(\frac{2^{k/2} \Gamma \left(\frac{k+1}{2}\right)}{\sqrt{\pi }} \right. \\
& \ \ \left. -\frac{\left(2^{\frac{k-1}{2}}   \left((k-1) \sqrt{\pi } \Gamma \left(\frac{k}{2}+1\right)-2 \Gamma   \left(\frac{k+1}{2}\right)\right)\right) }{3 \pi \sqrt{n}}+O\left(\frac{1}{n}\right)\right).
\end{align*}
More specifically, the expected length of such a path is
$$\mathbb{E}[X_{v, l}(n)] = \sqrt{\frac{2n}{\pi}} - \frac 2 {3\pi} + \bigo{\frac{1}{\sqrt{n}}}.$$
\end{example}





\section{Arbitrary Paths in Specific Tree Structures}
In some cases, the generating functions for vertical paths turn out be very simple, which makes it possible to calculate the expected length explicitly. In the case of general trees, it is also possible to find corresondences between vertical paths and general paths. We also present the case of binary trees to show that there is not always so strong a relationship. Further, the case of binary trees shows that it is not true for general paths that if one insists that a path has leaves as endpoints, it does not appreciably add to the expected length.

\subsection{Downward Paths in General Trees}
Let $\mathcal G$ be the class of general. Then the generating functions for trees, vertices, and leaves, respectively, are

\begin{equation*}
T_\mathcal{G}(x)=\frac{1-\sqrt{1-4x}}{2}, \quad V_\mathcal{G}(x)=\frac{x}{\sqrt{1-4x}}, \quad \text{and} \quad L_\mathcal{G}(x)=\frac{x}{2\sqrt{1-4x}}+\frac{x}{2}.
\end{equation*}

By Corollary 2.3 (1), with $k=1$, the generating function for the number of edges in paths from the root to a leaf in trees on $n$ vertices is

\begin{equation*}
\left(\frac{1}{2\sqrt{1-4x}}+\frac{1}{2} \right) \left(\frac{x}{2\sqrt{1-4x}}+\frac{x}{2}-x \right)=\dfrac{x^2}{1-4x}.
\end{equation*}

Undergraduate calculus shows that the number of edges in such paths is $4^{n-2}$. Since the number of such paths is clearly equal to the number of leaves, $\frac {1}{2} {2n-2 \choose n-1},$ \linebreak for $n \geq 2$ (each leaf has a unique path to the root), it follows that the expected length of such a path is

\begin{equation}
\frac{4^{n-2}}{\frac {1}{2} {2n-2 \choose n-1}}=\frac{4^{n-2}}{\frac{4^{n-1} (2n-3)!!}{2(2n-2)!!}}=\frac{(2n-2)!!}{2(2n-3)!!} \approx \frac{\sqrt{\pi n}}{2} + \bigo{\dfrac{1}{\sqrt{n}}}.
\end{equation}
A better asymptotic approximation for this number is given in the appendix. Since the expected distance to the closest leaf is approximately $1.62297 + \bigo{\dfrac{1}{n}}$ \cite{heuberger_prodinger_2017}, and the furthest leaf is $\sqrt{\pi n} - \frac{1}{2} + \bigo{\dfrac{1}{n}}$ \cite{de_bruijn_knuth_rice_1972}, the difference between the expected distance from the root to a leaf and the average of the distances to the closest and furthest leaves is only about $0.561486$ in an arbitrarily large tree.

By Corollary 2.3 (2), the generating function for the number of edges in downward paths in general trees is 
\begin{equation*}
\left(\frac{1}{2\sqrt{1-4x}} +\frac{1}{2}\right)^2 \left( \frac{x}{\sqrt{1-4x}}- \frac{1-\sqrt{1-4x}}{2}\right)=\frac{x^2}{(1-4x)^{3/2}},
\end{equation*}
and the generating function for the number of such paths is
\begin{equation*}
\left(\frac{1}{2\sqrt{1-4x}} +\frac{1}{2}\right) \left( \frac{x}{\sqrt{1-4x}}- \frac{1-\sqrt{1-4x}}{2}\right)=\frac{x}{2 (1-4 x)}-\frac{x}{2 \sqrt{1-4 x}}.
\end{equation*}
Both of these generating functions are straightforward enough to extract exact coefficients; 
\begin{align*}
[x^n] \frac{x^2}{(1-4x)^{3/2}}&=\frac{(2n-3)!}{(n-2)!^2}, \\
[x^n] \left( \frac{x}{2(1-4x)}-\frac{x}{2\sqrt{1-4x}} \right)&=\frac{1}{2}\left(4^{n-1} - {2n-2 \choose n-1} \right).
\end{align*}
 The expected length of a downward path in a general tree on $n$ vertices is the ratio of the two,

\begin{equation*}
\frac{\frac{(2n-3)!}{((n-2)!^2)}}{\frac{1}{2}\left(4^{n-1} - {2n-2 \choose n-1} \right)}=\frac{n-1}{\frac{(2n-2)!!}{(2n-3)!!}-1}.
\end{equation*}

A careful reader may note that the number of edges in downward paths is equal to ${n \choose 2} c_{n-1}$, where $c_{n-1}=\dfrac{1}{n} \displaystyle {2n-2 \choose n-1}$, the $(n-1)$st Catalan number, which is the number of trees on $n$ vertices. There is also the slightly more obvious fact that the generating function is one half times the second derivative of the generating function for trees. This is also the number of all paths, since one can choose any two endpoints and uniquely determine a path. These generating function arguments prove that their number is the same, but the fact invites a bijective proof.


\subsection{Arbitrary Paths in General Trees}
For this section, it will be useful to repeatedly refer to the edge between the root and its leftmost child. Thus, we will refer to this edge as the ``key edge.''
\begin{prop}
In general trees on $n$ vertices, the number of paths is equal to the number of edges in vertical paths.
\end{prop}
\begin{proof}
We construct a bijection between the set of marked edges in marked vertical paths to marked paths. Note that this bijection will not, in general, map into the same tree.

Take a path. If the path is vertical send it to the same path and mark the top edge. Else, mark the edge in the path on the right of the vertex closest to the root, and remove this edge and all of its siblings to its right. This gives a subtree where there is path through the key edge, and the key edge is marked. Graft this subtree onto the end of the left-hand path, to the right of the existing children; this gives a vertical path with a single marked edge.
\end{proof}

\begin{figure}[ht]
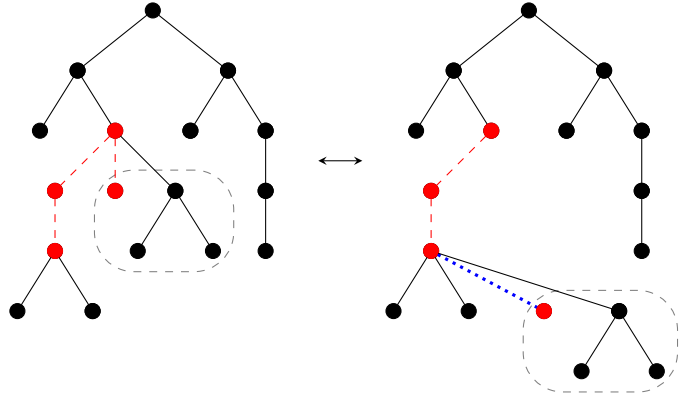

\begin{istgame}
\setistSolidNodeStyle{.2cm}
\xtdistance{8mm}{10mm}
\xtShowEndPoints
\istroot(A01)(0,0)
\istb
\istb <missing>
\istb
\endist
\istroot(A11)(A01-1)
\istb
\istb
\endist
\istroot(A12)(A01-3)
\istb
\istb
\endist
\xtdistance{8mm}{8mm}
\istroot(A22)(A11-2)[red]
\istb[red, dashed]
\istb[red, dashed]
\istb
\endist
\xtdistance{8mm}{10mm}
\istroot(A23)(A12-1)
\endist
\istroot(A24)(A12-2)
\istb
\endist
\istroot(A31)(A22-1)[red]
\istb[red, dashed]
\endist
\istroot(A32)(A22-2)[red]
\endist
\istroot(A33)(A22-3)
\istb
\istb
\endist
\istroot(A34)(A24-1)
\istb
\endist
\istroot(A41)(A31-1)[red]
\istb
\istb
\endist
\istroot(A52)(A41-2)
\endist
\xtSubgameBox(A32){(A32)(A33)(A33-1)(A33-2)}[inner sep = 5]
\draw[black, <->,
    shorten >=0.4pt]
    (2.2,-2) -- (2.8,-2);
\istroot(B01)(5,0)
\istb
\istb <missing>
\istb
\endist
\istroot(B11)(B01-1)
\istb
\istb
\endist
\istroot(B12)(B01-3)
\istb
\istb
\endist
\xtdistance{8mm}{8mm}
\istroot(B22)(B11-2)[red]
\istb[red, dashed]
\istb <missing>
\istb <missing>
\endist
\xtdistance{8mm}{10mm}
\istroot(B23)(B12-1)
\endist
\istroot(B24)(B12-2)
\istb
\endist
\istroot(B31)(B22-1)[red]
\istb[red, dashed]
\endist
\istroot(B34)(B24-1)
\istb
\endist
\istroot(B41)(B31-1)[red]
\istb<missing>
\istb<missing>
\istb
\istb
\istb[blue, dotted, very thick]
\istb
\endist
\istroot(B53)(B41-5)[red]
\endist
\istroot(B54)(B41-6)
\istb
\istb
\endist
\xtSubgameBox(B53){(B53)(B54)(B54-1)(B54-2)}[inner sep=5]
\end{istgame}
\caption{An example of the bijection used in Proposition 3.1.}
\end{figure}

\begin{prop}
In all general trees on $n$ vertices, the number of edges in all paths which include the key edge is ${n \choose 2} c_{n-1}$.
\end{prop}

\begin{proof}
We establish a bijection between all paths in trees on $n$ vertices to edges in paths which pass through the key edge. We will group all paths into one of four cases. There is overlap between cases 1 and 2, but the mappings agree.

\begin{enumerate}
\item[Case 1] Paths which contain the key edge: for any such path, map it to the same path in the same tree, and, within that path, mark the key edge. \\

\begin{figure}[ht]
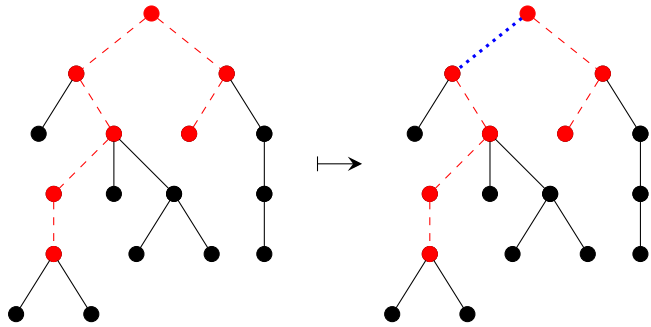

\begin{istgame}
\setistSolidNodeStyle{.2cm}
\xtdistance{8mm}{10mm}
\xtShowEndPoints
\istroot(10)(0,0)[red] 
\istb[red, dashed]
\istb<missing> 
\istb[red, dashed]
\endist 
\istroot(11)(10-1)[red]
\istb
\istb[red, dashed]
\endist
\istroot(12)(10-3)[red]
\istb[red, dashed]
\istb
\endist
\xtdistance{8mm}{8mm}
\istroot(13)(11-2)[red]
\istb[red, dashed] \istb \istb
\endist
\xtdistance{8mm}{10mm}
\istroot(14)(12-2)
\istb 
\endist
\istroot(15)(13-1)[red]
\istb[red, dashed]
\endist
\istroot(16)(13-3)
\istb
\istb
\endist
\istroot(17)(14-1)
\istb
\endist
\istroot(18)(15-1)[red]
\istb
\istb
\endist
\istroot(19)(12-1)[red]
\endist

\draw[black, |->,decoration={markings,mark=at position 1 with {\arrow[scale=1.7,black]{>}}},
    postaction={decorate},
    shorten >=0.4pt]
    (2.2,-2) -- (2.8,-2);

\istroot(20)(5,0)[red] 
\istb [blue, dotted, very thick]
\istb<missing> 
\istb[red, dashed]
\endist 
\istroot(21)(20-1)[red]
\istb
\istb[red, dashed]
\endist
\istroot(22)(20-3)[red]
\istb[red, dashed]
\istb
\endist
\xtdistance{8mm}{8mm}
\istroot(23)(21-2)[red]
\istb[red, dashed] \istb \istb
\endist
\xtdistance{8mm}{10mm}
\istroot(24)(22-2)
\istb 
\endist
\istroot(25)(23-1)[red]
\istb[red, dashed]
\endist
\istroot(26)(23-3)
\istb
\istb
\endist
\istroot(27)(24-1)
\istb
\endist
\istroot(28)(25-1)[red]
\istb
\istb
\endist
\istroot(29)(22-1)[red]
\endist
\end{istgame}
\caption{An example of case 1. The dotted edge in the dashed path on the right will be referred to as the \textit{key edge}.}
\end{figure}

\item[Case 2] Vertical paths: for any such path, map it to the unique path which passes from the key edge to the bottom edge of this path. This path passes through the root if the downward path is not contained in the subtree beneath the key edge. Next, mark the edge which was the edge of the original path which was closest to the root.\\

\begin{figure}[ht]
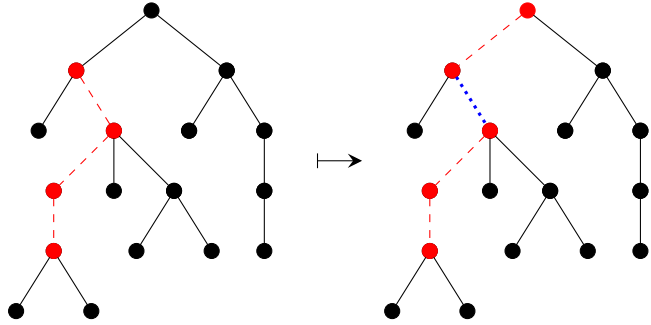

\begin{istgame}
\setistSolidNodeStyle{.2cm}
\xtdistance{8mm}{10mm}
\xtShowEndPoints
\istroot(10)(0,0) 
\istb 
\istb<missing> 
\istb 
\endist 
\istroot(11)(10-1)[red]
\istb
\istb[red, dashed]
\endist
\istroot(12)(10-3)
\istb
\istb
\endist
\xtdistance{8mm}{8mm}
\istroot(13)(11-2)[red]
\istb[red, dashed] \istb \istb
\endist
\xtdistance{8mm}{10mm}
\istroot(14)(12-2)
\istb 
\endist
\istroot(15)(13-1)[red]
\istb[red, dashed]
\endist
\istroot(16)(13-3)
\istb
\istb
\endist
\istroot(17)(14-1)
\istb
\endist
\istroot(18)(15-1)[red]
\istb
\istb
\endist
\istroot(19)(12-1)
\endist

\draw[black, |->,decoration={markings,mark=at position 1 with {\arrow[scale=1.7,black]{>}}},
    postaction={decorate},
    shorten >=0.4pt]
    (2.2,-2) -- (2.8,-2);

\istroot(20)(5,0)[red] 
\istb[red, dashed]
\istb<missing> 
\istb 
\endist 
\istroot(21)(20-1)[red]
\istb
\istb[blue, dotted, very thick]
\endist
\istroot(22)(20-3)
\istb
\istb
\endist
\xtdistance{8mm}{8mm}
\istroot(23)(21-2)[red]
\istb[red, dashed] \istb \istb
\endist
\xtdistance{8mm}{10mm}
\istroot(24)(22-2)
\istb 
\endist
\istroot(25)(23-1)[red]
\istb[red, dashed]
\endist
\istroot(26)(23-3)
\istb
\istb
\endist
\istroot(27)(24-1)
\istb
\endist
\istroot(28)(25-1)[red]
\istb
\istb
\endist
\istroot(29)(22-1)
\endist
\end{istgame}
\caption{An example of case 2.}
\end{figure}

\item[Case 3] Non-vertical paths which are contained in the subtree below the key edge: consider the vertex closest to the root; remove the subtree whose key edge is the right-hand edge of the path. This leaves two trees with vertical paths, the second with the path through its key edge (this is necessary for the reversal of the algorithm). Mark the top edge of the path in the first tree. Graft the second tree to the root of the first tree, to the right of all existing edges. Finally, connect the two paths \\

\begin{figure}[ht]
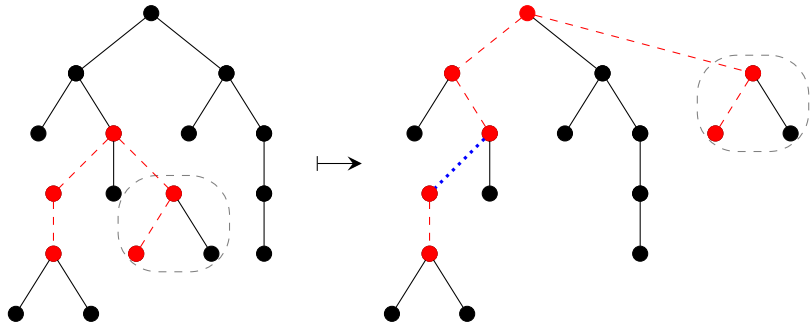

\begin{istgame}
\setistSolidNodeStyle{.2cm}
\xtdistance{8mm}{10mm}
\xtShowEndPoints
\istroot(10)(0,0) 
\istb 
\istb<missing> 
\istb 
\endist 
\istroot(11)(10-1)
\istb
\istb
\endist
\istroot(12)(10-3)
\istb
\istb
\endist
\xtdistance{8mm}{8mm}
\istroot(13)(11-2)[red]
\istb[red, dashed] \istb \istb[red, dashed]
\endist
\xtdistance{8mm}{10mm}
\istroot(14)(12-2)
\istb 
\endist
\istroot(15)(13-1)[red]
\istb[red, dashed]
\endist
\istroot(16)(13-3)[red]
\istb[red, dashed]
\istb
\endist
\istroot(17)(14-1)
\istb
\endist
\istroot(18)(15-1)[red]
\istb
\istb
\endist
\istroot(19)(16-1)[red]
\endist
\xtSubgameBox(16){(13-3)(16-2)(16-1)}[inner sep = 4]

\draw[black, |->,decoration={markings,mark=at position 1 with {\arrow[scale=1.7,black]{>}}},
    postaction={decorate},
    shorten >=0.4pt]
    (2.2,-2) -- (2.8,-2);

\istroot(20)(5,0)[red] 
\istb<missing>
\istb <missing>
\istb[red, dashed]
\istb <missing>
\istb
\istb<missing>
\istb[red, dashed]
\endist 
\istroot(21)(20-3)[red]
\istb
\istb[red, dashed]
\endist
\istroot(22)(20-5)
\istb
\istb
\endist
\xtdistance{8mm}{8mm}
\istroot(23)(21-2)[red]
\istb[blue, dotted, very thick] \istb \istb<missing>
\endist
\xtdistance{8mm}{10mm}
\istroot(24)(22-2)
\istb 
\endist
\istroot(25)(23-1)[red]
\istb[red, dashed]
\endist
\istroot(26)(20-7)[red]
\istb[red, dashed]
\istb
\endist
\istroot(27)(24-1)
\istb
\endist
\istroot(28)(25-1)[red]
\istb
\istb
\endist
\istroot(29)(26-1)[red]
\endist
\xtSubgameBox(26){(20-7)(26-1)(26-2)}[inner sep = 4]
\end{istgame}
\caption{An example of case 3.}
\end{figure}

\item[Case 4] Non-downward paths which are not contained in the subtree below the key edge: similarly to case 3, take the vertex closest to the root and remove the subtree whose key edge is the first edge of the right-hand path. Mark the top edge of the path that remains, then graft the second tree at the bottom of the key edge, to the right of existing vertices, then connect the paths. \\

\begin{figure}[ht]
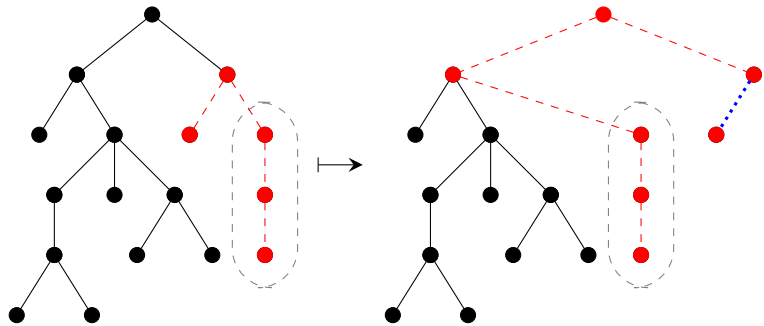

\begin{istgame}
\setistSolidNodeStyle{.2cm}
\xtdistance{8mm}{10mm}
\xtShowEndPoints
\istroot(10)(0,0) 
\istb 
\istb<missing> 
\istb 
\endist 
\istroot(11)(10-1)
\istb
\istb
\endist
\istroot(12)(10-3)[red]
\istb[red, dashed]
\istb[red, dashed]
\endist
\xtdistance{8mm}{8mm}
\istroot(13)(11-2)
\istb \istb \istb
\endist
\xtdistance{8mm}{10mm}
\istroot(14)(12-2)[red]
\istb[red, dashed]
\endist
\istroot(15)(13-1)
\istb
\endist
\istroot(16)(13-3)
\istb
\istb
\endist
\istroot(17)(14-1)[red]
\istb[red, dashed]
\endist
\istroot(18)(15-1)
\istb
\istb
\endist
\istroot(19)(12-1)[red]
\endist
\istroot(101)(17-1)[red]
\endist
\xtSubgameBox(14){(14)(101)}[inner sep = 9.5]

\draw[black, |->,decoration={markings,mark=at position 1 with {\arrow[scale=1.7,black]{>}}},
    postaction={decorate},
    shorten >=0.4pt]
    (2.2,-2) -- (2.8,-2);

\istroot(20)(6,0)[red] 
\istb[red, dashed]
\istb<missing> 
\istb <missing>
\istb <missing>
\istb[red, dashed]
\endist 
\istroot(21)(20-1)[red]
\istb <missing>
\istb <missing>
\istb
\istb
\istb<missing>
\istb[red, dashed]
\endist
\istroot(22)(20-5)[red]
\istb[blue, dotted, very thick]
\istb<missing>
\endist
\xtdistance{8mm}{8mm}
\istroot(23)(21-4)
\istb \istb \istb
\endist
\xtdistance{8mm}{10mm}
\istroot(24)(21-6)[red]
\istb[red, dashed]
\endist
\istroot(25)(23-1)
\istb
\endist
\istroot(26)(23-3)
\istb
\istb
\endist
\istroot(27)(24-1)[red]
\istb[red, dashed]
\endist
\istroot(28)(25-1)
\istb
\istb
\endist
\istroot(29)(22-1)[red]
\endist
\istroot(201)(27-1)[red]
\endist
\xtSubgameBox(24){(24)(201)}[inner sep = 9.5]
\end{istgame}
\caption{An example of case 4.}
\end{figure}

\end{enumerate}
\end{proof}

\begin{corollary}
The number of edges in all paths, in trees on $n$ vertices is \linebreak $(n-1)4^{n-2}$. Further, the expected length of a uniformly selected path in a tree of size $n$ is equal to the expected length of a uniformly selected path from the root to a leaf in trees of size $n$.
\end{corollary}

\begin{proof}
We may decompose any path with a marked edge into a tree with a marked vertex and a tree with a path through the key edge with a marked edge. This is done as follows: take the vertex which is closest to the root contained in the path and remove the subtree starting at the left edge of the path.

\begin{figure}
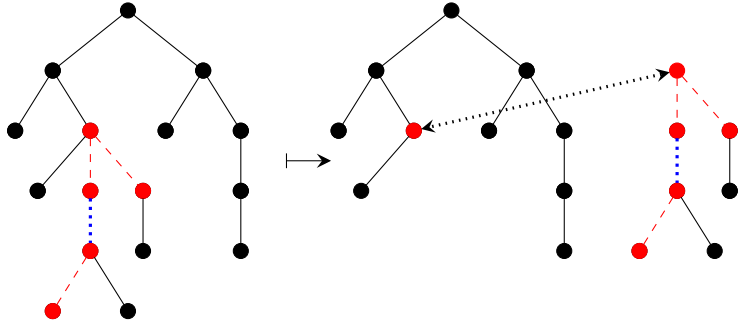

\begin{istgame}
\setistSolidNodeStyle{.2cm}
\xtdistance{8mm}{10mm}
\xtShowEndPoints
\istroot(A0)(0,0)
\istb
\istb<missing>
\istb
\endist
\istroot(A11)(A0-1)
\istb
\istb
\endist
\istroot(A12)(A0-3)
\istb
\istb
\endist
\istroot(A21)(A11-1)
\endist
\xtdistance{8mm}{7mm}
\istroot(A22)(A11-2)[red]
\istb
\istb[red, dashed]
\istb[red, dashed]
\endist
\xtdistance{8mm}{10mm}
\istroot(A23)(A12-1)
\endist
\istroot(A24)(A12-2)
\istb
\endist
\istroot(A31)(A22-1)
\endist
\istroot(A32)(A22-2)[red]
\istb[blue, dotted, very thick]
\endist
\istroot(A33)(A22-3)[red]
\istb
\endist
\istroot(A34)(A24-1)
\istb
\endist
\istroot(A42)(A32-1)[red]
\istb[red, dashed]
\istb
\endist 
\istroot(A51)(A42-1)[red]
\endist

\draw[black, |->,decoration={markings,mark=at position 1 with {\arrow[scale=1.7,black]{>}}},
    postaction={decorate},
    shorten >=0.4pt]
    (2.1,-2) -- (2.7,-2);

\istroot(B0)(4.3,0)
\istb
\istb<missing>
\istb
\endist
\istroot(B11)(B0-1)
\istb
\istb
\endist
\istroot(B12)(B0-3)
\istb
\istb
\endist
\istroot(B21)(B11-1)
\endist
\xtdistance{8mm}{7mm}
\istroot(B22)(B11-2)[red]
\istb
\istb<missing>
\istb<missing>
\endist
\xtdistance{8mm}{10mm}
\istroot(B23)(B12-1)
\endist
\istroot(B24)(B12-2)
\istb
\endist
\istroot(B34)(B24-1)
\istb
\endist

\xtdistance{8mm}{7mm}
\istroot(C0)(7.3,-0.8)[red]
\istb<missing>
\istb[red, dashed]
\istb[red, dashed]
\endist
\xtdistance{8mm}{10mm}
\istroot(C11)(C0-2)[red]
\istb[blue, dotted, very thick]
\endist
\istroot(C12)(C0-3)[red]
\istb
\endist
\istroot(C21)(C11-1)[red]
\istb[red, dashed]
\istb
\endist
\istroot(C22)(C12-1)
\endist
\istroot(C31)(C21-1)[red]
\endist
\xtInfoset[<->, very thick](C0)(B22)
\end{istgame}
\caption{The bijection corresponding to (2.3).}
\end{figure}
Let $E_\mathcal{G}(x)$ be the generating function for the number of edges in all paths on $n$ vertices. Then, by the product formula for generating functions, 
\begin{equation}
E_\mathcal{G}(x) = \frac 1 x V_\mathcal{G}(x) \cdot \frac{x^2}{(1-4x)^{3/2}}= \frac{x^2}{(1-4x)^2},
\end{equation}
and the result follows by the binomial theorem.

Thus the expected length of a uniformly randomly selected path in a uniformly randomly selected ordered tree on $n$ vertices is

\begin{equation}
\dfrac{(n-1)4^{n-2}}{{n \choose 2} \frac{1}{n}{2n-2 \choose n-1}}=\dfrac{(n-1)4^{n-2}}{\frac{(n-1)}{2} \frac{4^{n-1}(2n-3)!!}{(2n-2)!!}}=\frac{(2n-2)!!}{2(2n-3)!!}.
\end{equation}
Thus, for all $n \geq 2$, the expected length of a uniformly selected path is equal to the expected length of a uniformly selected path from the root to a leaf (comparing with (3.1)).
\end{proof}

The sum of all distances in a graph is also called the Wiener Index of a graph. The Wiener Index for general trees was previously derived by Entringer, Meir, Moon, and Sz{\'e}kely using different methods, including a nice bijective proof \cite{entringer_meir_moon_szekely_1994}.

\begin{prop}
The number of paths between leaves in all general trees on $n$ vertices is 
$$\dfrac{(2n-5)!}{((n-3)!)^2},$$
and the number of edges in all such paths is 
$$(n-2)4^{n-3}+  \dfrac{(2n-5)!}{(n-3)!^2}.$$
\end{prop}

\begin{proof}
To find a path between two leaves, we first consider the generating function for the number of such paths which include the root. If we remove the root of a tree with such a path we have a sequence of trees, followed by a tree with a marked leaf, followed by another sequence, another tree with a marked leaf, and a third sequence. Thus, since $T_\mathcal{G}(x) = x (1-T_\mathcal{G}(x))^{-1},$ the generating function for the number of trees with two marked leaves is

$$x \left(\frac{1}{1-T_\mathcal{G}(x)}\right)^3 L^2_\mathcal{G} (x) = \frac{T_\mathcal{G}^3(x)}{x^2} L_\mathcal{G}^2(x).$$

Using the Grafting Lemma to pass to any such path, and substituting using (2.1), the generating function for the number of paths which ends in leaves 

\begin{equation*}
\frac{L_\mathcal{G}(x)}{x} \frac{T_\mathcal{G}^3(x)}{x^2} L_\mathcal{G}^2(x)=V_\mathcal{G}^3(x)=\frac{x^3}{(1-4x)^{3/2}}.
\end{equation*}

Extracting coefficients gives the first part of the statement. For the second part, we appeal to the bijection used for Proposition 2.4. First, it is not hard to see, by examination of the cases, that this bijection maps paths whose endpoints are leaves to edges in such paths. However, not all edges are the image of such a path. The edges which are missed are exactly the key edges which end in leaves. These fall under case 2 of the bijection, so the preimage of such an edge is a vertical path which is not under the key edge, where the key edge of the tree ends in a leaf. We may attain any such path uniquely by taking a tree with a downward path which terminates at a leaf, and adding a single vertex to the root, on the left. Thus, the generating function for such cases is 

$$x \frac{L_\mathcal{G}(x)}{x} \left( L_\mathcal{G}(x) - x\right)=\frac{x^3}{1-4x}.$$

By the argument used Corollary 3.3, the generating function for edges in paths from leaves is

$$ \frac{V_\mathcal{G}(x)}{x} \left(\frac{x^3}{(1-4x)^{3/2}}+\frac{x^3}{1-4x} \right)= \frac{x^3}{(1-4x)^{2}}+\frac{x^3}{(1-4x)^{3/2}},$$

and the result follows by extracting coefficients.
\end{proof}

This implies that the number of edges in paths between leaves in trees on $n$ vertices is equal to the number of edges in all paths plus the number of paths, or, equivalently, the number of all marked vertices in all marked paths in trees on $n-1$ vertices. This result also invites a bijective proof, which is not known to the author.

\begin{corollary}
The expected length of a uniformly selected path between leaves in a uniformly selected general tree on $n$ vertices is 
\begin{equation}
\dfrac{(2n-4)!!}{2(2n-5)!!}+1 \approx \dfrac{\sqrt{\pi n}}{2}+1- \frac{7 \sqrt{\pi}}{16\sqrt{n}} + \bigo{\frac{1}{n^{3/2}}}.
\end{equation}
\end{corollary}

The exact expression is attained by taking ratios, and the asymptotic expression is attained by using singularity analysis on the generating functions and once again taking ratios. Thus, insisting that a path end in leaves only adds a constant to the expected value, as it is with vertical paths in all graftable trees satisfying (2.2).

\subsection{Arbitrary Paths in Binary Trees}

Let $\mathcal{B}$ be the class of binary trees. First, let $p_{n,k}$ be the number of paths in trees on $n$ vertices which originate at the root and end at a a leaf, and let $$P_\mathcal{B}(x,y)= \displaystyle \sum_{n, k, \geq 1} p_{n,k}x^n y^k.$$  
If we take such a tree, either it is a single vertex, or we can remove the root, leaving a marked edge and a similar tree in one of two ways (since left or right matters), or we may have a similar tree and a regular binary tree in one of two ways. Thus 
$$P_\mathcal{B}(x,y)= x+2 x y P_\mathcal{B}(x,y)+ 2 x y P_\mathcal{B}(x,y) T_\mathcal{B}(x),$$ 
with $T_\mathcal{B}(x)= \dfrac{1-2x-\sqrt{1-4x}}{2x}$. Solving for $P(x,y),$

\begin{equation*}
P(x,y)=\frac{x}{1-y+y \sqrt{1-4 x}}.
\end{equation*}

With this in hand, we can find the bivariate generating function for all paths in a tree using the grafting lemma. First, we graft a tree onto the end of our marked path so that the generating function for paths that start at the root and end anywhere is $\frac{P(x,y)}{x} \cdot T_\mathcal{B}(x)$. The generating function for a path that passes through the root is $x y^2 \left( \frac{P(x,y)}{x} \cdot T_\mathcal{B}(x)\right)^2$, since, if we remove the root, the tree must have two marked edges leading to two children, both of which are a tree with a marked path from the root. The sum of these two gives all paths whose top vertex is the root. We then apply the grafting lemma to give paths whose top vertex is anywhere in the tree. Let $p^*_{n,k}$ be the number of all paths in binary trees on $n$ vertices with $k$ edges, and let
$$P^*_\mathcal{B}(x,y)= \displaystyle \sum_{n, k, \geq 1} p^*_{n,k}x^n y^k.$$  
Then 
\begin{equation*}
P^*(x,y)=\frac{P(x,y)}{x} \cdot T_\mathcal{B}(x)+x y^2 \left( \frac{P(x,y)}{x} \cdot T_\mathcal{B}(x)\right)^2.
\end{equation*}
Then the generating function for the number of edges of trees on size $n$ is 
\begin{equation*}
\left. \frac{\partial}{\partial y} P^*(x,y) \right|_{y=1}=\frac{1-\sqrt{1-4x}-2x}{(1-4x)^2}.
\end{equation*}
The generating function for the number of such paths is simply $\frac{x^2}{2} \frac{d^2}{dx^2}T_B(x),$ so, applying singularity analysis in both cases and taking ratios, we have that the expected length of an a uniformly selected path in a binary tree on $n$ vertices is
\begin{equation*}
\sqrt{\pi n} -4+\bigo{\dfrac{1}{n^{1/2}}}.
\end{equation*}

To find the generating functions for arbitrary paths which end in leaves, we carry out the same process, but omit the multiplication by $T_\mathcal{B}(x)$ at the end of the paths, which gives the generating functions and results supplied in the appendix. It is straightforward to show that the generating function for the number of paths between leaves is $x (L_\mathcal{B}(x))^2.$ Comparing expectations shows that paths between leaves are expected to $\frac 4 \pi \approx 1.2732$ times longer than arbitrary paths as the size of the tree goes to infinity, in contrast to general trees, where the expected lengths differ by a constant.
\section{Appendix}

\subsection{Calculated results}
All formulas on the following tables should be taken for $n \geqslant 2$, or in the cases of paths between leaves, for $n \geqslant 3$, since there are no paths on trees of lesser size. Thus, some of the formulas do not match the coefficients for the generating functions for smaller values of $n$.

\begin{table}
\makebox[1 \textwidth][c]{ 
\small
\def\arraystretch{2}
\begin{tabular}{|m{.8cm} | c | c | m{8cm} |}
\hline
\multicolumn{4}{|c|}{General Trees} \\
\hline
\multirow{6}{0.8cm}{\hspace{0.1in}\begin{sideways}  Downward Paths from The Root  \hspace{0.45 in} \end{sideways}} 
& \multirow{3.3}{0.5in}{Number} 
& Generating Function & \vspace{-.05in} $\displaystyle\dfrac{x}{\sqrt{1-4x}}-\dfrac{1-\sqrt{1-4x}}{2}$ \\ 
\cline{3-4}
& & Coefficient & \vspace{0.02in} $\displaystyle{2n-2 \choose n} $\\
\cline{3-4}
& &Asymptotic Coefficient & $4^n \left( \dfrac{1}{4 \sqrt{\pi n}}-\dfrac{5}{32 \sqrt{\pi n^3}} - \dfrac{23}{512 \sqrt{\pi n^5}} \right) +\bigo{\dfrac{4^n }{n^{7/2}}}$ \\ 
\cline{2-4}
& \multirow{3.3}{0.5in}{Edges}
& Generating Function & \vspace{-.05in} $\dfrac{x}{2(1-4x)} - \dfrac{x}{2\sqrt{1-4x}}$ \\ 
\cline{3-4}
& & Coefficient &\vspace{0.02in}$\displaystyle\dfrac{1}{2} \left(4^{n-1}-{2n-2 \choose n-1} \right)$  \\
\cline{3-4}
& &Asymptotic Coefficient & $4^n \left( \dfrac{1}{8} - \dfrac{1}{8 \sqrt{ \pi n}} - \dfrac{3}{64\sqrt{\pi n^3}} \right) + \bigo{\dfrac{4^n}{n^{5/2}}}$ \\
\cline{2-4}
& \multirow{2.2}{0.5 in}{Expected Length} 
&Coefficient & \vspace{-.05in}$n\left(\dfrac{(2n-4)!!}{(2n-3)!!}-\dfrac{1}{2(n-1)}\right) $\\
\cline{3-4}
& & Asymptotic Coefficient& $\dfrac{\sqrt{\pi n}}{2}-\dfrac{1}{2} + \dfrac{5 \sqrt{\pi}}{16 \sqrt{n}}- \dfrac{1}{2n} +\dfrac{73 \sqrt{\pi}}{256 \sqrt{n^3}} -\dfrac{1}{2n^2}+ \bigo{\dfrac{1}{n^{5/2}}}$ \\ 
\hline
\multirow{6}{1cm}{\hspace{0.1in}\begin{sideways}  Downward Paths  \hspace{0.85 in} \end{sideways}} 
& \multirow{3.3}{0.5in}{Number} 
& Generating Function & $\dfrac{x}{2(1-4x)} - \dfrac{x}{2\sqrt{1-4x}}$ \\ 
\cline{3-4}
& & Coefficient &\vspace{0.02in} $\displaystyle \dfrac{1}{2} \left(4^{n-1}-{2n -2\choose n-1} \right)$ \\
\cline{3-4}
& &Asymptotic Coefficient & $4^n \left( \dfrac{1}{8} - \dfrac{1}{8 \sqrt{ \pi n}} - \dfrac{3}{64\sqrt{\pi n^3}} \right) + \bigo{\dfrac{4^n}{n^{5/2}}}$ \\
\cline{2-4}
& \multirow{3.3}{0.5in}{Edges}
& Generating Function & $\dfrac{x^2}{(1-4x)^{3/2}}$ \\ 
\cline{3-4}
& & Coefficient & $\dfrac{(2n-3)!}{((n-2)!)^2}$ \\
\cline{3-4}
& &Asymptotic Coefficient & $4^n \left( \dfrac{\sqrt{n}}{8 \sqrt{\pi}}-\dfrac{5}{64 \sqrt{\pi n}} - \dfrac{23}{1024 \sqrt{\pi n ^3}}\right) + \bigo{\dfrac{4^n}{n^{5/2}}}$ \\
\cline{2-4}
& \multirow{2.2}{0.5 in}{Expected Length} 
&Coefficient & $\dfrac{n-1}{\frac{(2n-2)!!}{(2n-3)!!}-1}$\\
\cline{3-4}
& & Asymptotic Coefficient& $\dfrac{\sqrt{n}}{ \sqrt{\pi}} + \dfrac{1}{\pi}  + \dfrac{8-5\pi}{8 \pi^{3/2} \sqrt{n}} +\dfrac{4-\pi}{4\pi^2 n} +\bigo{\dfrac{1}{n^{3/2}}}$ \\ 
\hline
\multirow{6}{1cm}{\begin{sideways}  \parbox{1.8in}{\begin{center}Downward Paths from The Root Which End in Leaves\end{center}}  \hspace{0.325in}  \end{sideways}} 
& \multirow{3.3}{0.5in}{Number} 
& Generating Function & $\dfrac{x}{2\sqrt{1-4x}}-\dfrac{x}{2}$ \\
\cline{3-4}
& & Coefficient &\vspace{0.02in} $\displaystyle \dfrac{1}{2} {2n-2 \choose n-1}$ \\
\cline{3-4}
& & Asymptotic Coefficient & $4^n \left( \dfrac{1}{8 \sqrt{\pi n}}+\dfrac{3}{64 \sqrt{\pi n^3}} + \dfrac{25}{1024 \sqrt{\pi n^5}} \right) +\bigo{\dfrac{4^n}{n^{7/2}}}$ \\
\cline{2-4}
& \multirow{2.2}{0.5in}{Edges}
& Generating Function & $\dfrac{x^2}{1-4x}$ \\
\cline{3-4}
& & Coefficient & $4^{n-2}$ \\
\cline{2-4}
& \multirow{2.2}{0.5 in}{Expected Length} 
&Coefficient & $\dfrac{(2n-2)!!}{2(2n-3)!!}$\\
\cline{3-4}
& & Asymptotic Coefficient& $\dfrac{\sqrt{\pi n}}{2} - \dfrac{3 \sqrt{\pi}}{16 \sqrt{n}}- \dfrac{7 \sqrt{\pi}}{256 \sqrt{n^3}} + \bigo{\dfrac{1}{n^{5/2}}}$ \\
\hline


\end{tabular}
}
\end{table}

\begin{table}
\makebox[1 \textwidth][c]{ 
\small
\def\arraystretch{2}
\begin{tabular}{| m{.8cm} | c | c | m{8cm} |}
\hline
\multicolumn{4}{|c|}{General Trees, cont.} \\
\hline
\multirow{6}{0.5in}{\begin{sideways}  \parbox{1.5in}{\begin{center}Downward Paths Which End in Leaves\end{center}}  \hspace{0.335in} \end{sideways}} 
& \multirow{2.2}{0.5in}{Number} 
& Generating Function &\vspace{-.05in} $\dfrac{x^2}{1-4x}$ \\
\cline{3-4}
& & Coefficient & $4^{n-2}$ \\
\cline{2-4}
& \multirow{3.3}{0.5in}{Edges}
& Generating Function & $\dfrac{x^2}{2(1-4x)^{3/2}}+\dfrac{x^2}{2(1-4x)}$ \\
\cline{3-4}
& & Coefficient & $\dfrac{(2n-3)!}{2((n-2)!)^2}+\dfrac{4^{n-2}}{2}$ \\
\cline{3-4}
& &Asymptotic Coefficient & $4^n \left( \dfrac{\sqrt{n}}{16 \sqrt{\pi}}+\dfrac{1}{32} -\dfrac{5}{128 \sqrt{\pi n}} -\dfrac{23}{2048 \sqrt{\pi n^3}}
\right) + \bigo{\dfrac{4^n}{n^{5/2}}}$ \\
\cline{2-4}
& \multirow{2.2}{0.5 in}{Expected Length} 
&Coefficient & $\dfrac{(2n-3)!!}{2(2n-4)!!}+\dfrac{1}{2}$\\
\cline{3-4}
& & Asymptotic Coefficient& $\dfrac{\sqrt{n}}{\sqrt{\pi}}+\dfrac {1}{2} - \dfrac{5}{8 \sqrt{\pi n}} - \dfrac{23}{128 \sqrt{\pi n^3}}+\bigo{\dfrac{1}{n^{5/2}}}$ \\
\hline
\multirow{6}{0.5in}{\hspace{.2cm}\begin{sideways}  \parbox{1.8in}{\begin{center} Arbitrary Paths \end{center}}  \hspace{0.25in} \end{sideways}} 
& \multirow{3.3}{0.5in}{Number} 
& Generating Function & $\dfrac{x^2}{(1-4x)^{3/2}}$ \\
\cline{3-4}
& & Coefficient & $\dfrac{(2n-3)!}{((n-2)!)^2}$ \\
\cline{3-4}
& & Asymptotic Coefficient & $4^n \left( \dfrac{\sqrt{n}}{8 \sqrt{\pi}}-\dfrac{5}{64 \sqrt{\pi n}} - \dfrac{23}{1024 \sqrt{\pi n ^3}}\right) + \bigo{\dfrac{4^n}{n^{5/2}}}$ \\
\cline{2-4}
& \multirow{2.2}{0.5in}{Edges}
& Generating Function & $\dfrac{x^2}{(1-4x)^2}$ \\
\cline{3-4}
& & Coefficient & $(n-1)4^{n-2}$ \\
\cline{2-4}
& \multirow{2.2}{0.5 in}{Expected Length} 
&Coefficient & $\dfrac{(2n-2)!!}{2(2n-3)!!}$\\
\cline{3-4}
& & Asymptotic Coefficient& $\dfrac{\sqrt{\pi n}}{2} - \dfrac{3 \sqrt{\pi}}{16 \sqrt{n}}- \dfrac{7 \sqrt{\pi}}{256 \sqrt{n^3}} + \bigo{\dfrac{1}{n^{5/2}}}$ \\
\hline
\multirow{6}{0.5in}{\hspace{.2cm}\begin{sideways}  \parbox{1.8in}{\begin{center} Paths Between Leaves \end{center}}  \hspace{0.475in} \end{sideways}} 
& \multirow{3.3}{0.5in}{Number} 
& Generating Function & $\dfrac{x^3}{(1-4x)^{3/2}}$ \\
\cline{3-4}
& & Coefficient & $\dfrac{(2n-5)!}{((n-3)!)^2}$ \\
\cline{3-4}
& & Asymptotic Coefficient & $\displaystyle 4^n \left( \frac{\sqrt{n}}{32\sqrt{\pi}}-\frac{9}{256\sqrt{\pi n}}- \frac{79}{4096 \sqrt{\pi n^3}} 
\right) + \bigo{\frac{4^n}{n^{5/2}}} $ \\
\cline{2-4}
& \multirow{3.3}{0.5in}{Edges}
& Generating Function &$\dfrac{x^3}{(1-4x)^2}+\dfrac{x^3}{(1-4x)^{3/2}}$ \\
\cline{3-4}
& & Coefficient & $\displaystyle (n-2)4^{n-3}+\dfrac{(2n-5)!}{((n-3)!)^2}$ \\
\cline{3-4}
& & Asymptotic Coefficient & $\displaystyle 4^n \left( \frac{n}{64}+\frac{\sqrt{n}}{32\sqrt{\pi}}-\frac{1}{32}-\frac{9}{256\sqrt{\pi n}}
\right) + \bigo{\frac{4^n}{n^{3/2}}} $ \\
\cline{2-4}
& \multirow{2.2}{0.5 in}{Expected Length} 
&Coefficient & $\dfrac{(2n-4)!!}{2(2n-5)!!}+1$\\
\cline{3-4}
& & Asymptotic Coefficient& $\displaystyle \frac{\sqrt{\pi n}}{2} +1 - \frac{7 \sqrt{\pi}}{16 \sqrt{n}} - \frac{47\sqrt{\pi}}{256 n ^{3/2}} + \bigo{\frac{1}{n^{5/2}}}$ \\
\hline
\end{tabular}
}
\end{table} 

\begin{table}
\makebox[1 \textwidth][c]{ 
\small
\def\arraystretch{2}
\begin{tabular}{|m{.8cm} | c | c | m{8cm} |}
\hline
\multicolumn{4}{|c|}{Binary Trees} \\
\hline
\multirow{6}{0.5in}{\hspace{0.1in}\begin{sideways}  Downward Paths from The Root  \hspace{0.25 in} \end{sideways}} 
& \multirow{3.3}{0.5in}{Number} 
& Generating Function & $1-\dfrac{1}{x}-\dfrac{3}{\sqrt{1-4x}}+\dfrac{1}{x\sqrt{1-4x}}$ \\ 
\cline{3-4}
& & Coefficient &\vspace{0.02in} $\displaystyle {2n+2 \choose n+1}- 3 {2n \choose n}$ \\
\cline{3-4}
& &Asymptotic Coefficient & $\displaystyle 4^n \left(\dfrac{1}{\sqrt{\pi n}} - \dfrac{17}{8 \sqrt{\pi n^3}} + \dfrac{289}{128 \sqrt{\pi n^5}} \right) +\bigo{\dfrac{4^n }{n^{7/2}}}$ \\
\cline{2-4}
& \multirow{3.3}{0.5in}{Edges}
& Generating Function & $\dfrac{1}{x(1-4x)}-\dfrac{1}{x\sqrt{1-4x}}-\dfrac{3}{1-4x}+\dfrac{1}{\sqrt{1-4x}}$ \\ 
\cline{3-4}
& & Coefficient & \vspace{0.02in} $\displaystyle 4^n - {2n+2 \choose n+1} + {2n \choose n}$  \\
\cline{3-4}
& &Asymptotic Coefficient & $4^n \left(1- \dfrac{3}{\sqrt{\pi n}}+\dfrac{19}{8 \sqrt{\pi n^3}}- \dfrac{291}{128 \sqrt{\pi n^5}} \right) + \bigo{\dfrac{4^n}{n^{7/2}}}$ \\
\cline{2-4}
& \multirow{1.2}{0.5 in}{Expected Length} 
& Asymptotic Coefficient& $\sqrt{\pi n}-3 + \dfrac{17 \sqrt{\pi}}{8\sqrt{n}}-\dfrac{4}{n}+\dfrac{289 \sqrt{\pi}}{128 \sqrt{n^3}}-\dfrac{4}{n^2}+ \bigo{\dfrac{1}{n^{5/2}}}$ \\
\hline
\multirow{6}{0.5in}{\hspace{0.1in}\begin{sideways}  Downward Paths  \hspace{0.6 in} \end{sideways}} 
& \multirow{3.3}{0.5in}{Number} 
& Generating Function & \vspace{-.05in} $\dfrac{1}{x(1-4x)}-\dfrac{1}{x\sqrt{1-4x}}-\dfrac{3}{1-4x}+\dfrac{1}{\sqrt{1-4x}}$ \\ 
\cline{3-4}
& & Coefficient & \vspace{0.02in} $ \displaystyle 4^n - {2n+2 \choose n+1} + {2n \choose n}$  \\
\cline{3-4}
& &Asymptotic Coefficient & $4^n \left(1- \dfrac{3}{\sqrt{\pi n}}+\dfrac{19}{8 \sqrt{\pi n^3}}- \dfrac{291}{128 \sqrt{\pi n^5}} \right) + \bigo{\dfrac{4^n}{n^{7/2}}}$ \\
\cline{2-4}
& \multirow{3.3}{0.5in}{Edges}
& Generating Function & \vspace{-.05in} $\dfrac{1}{x(1-4x)^{3/2}}-\dfrac{1}{x(1-4x)}-\dfrac{3}{(1-4x)^{3/2}}+\dfrac{1}{1-4x}$ \\ 
\cline{3-4}
& & Coefficient & $\dfrac{(2n+3)!}{((n+1)!)^2}-\dfrac{3(2n+1)!}{(n!)^2}-3\cdot4^n$ \\
\cline{3-4}
& &Asymptotic Coefficient & $4^n \left( \dfrac{2 \sqrt{n}}{\sqrt{\pi}} - 3 + \dfrac{19}{4 \sqrt{\pi n}} - \dfrac{167}{64 \sqrt{\pi n^3}} 
 \right) + \bigo{\dfrac{4^n}{n^{5/2}}}$ \\
\cline{2-4}
& \multirow{1.2}{0.5 in}{Expected Length} 
& Asymptotic Coefficient& $\dfrac{2\sqrt{n}}{\sqrt{\pi}}-3+\dfrac{6}{\pi}+\dfrac{72-17\pi}{4 \pi^{3/2} \sqrt{n}} +\dfrac{108-35\pi}{2 \pi^2 n}+\bigo{\dfrac{1}{n^{3/2}}}$ \\
\hline
\multirow{6}{0.5in}{\begin{sideways}  \parbox{1.8in}{\begin{center}Downward Paths from The Root Which End in Leaves\end{center}}  \hspace{0.45in}  \end{sideways}} 
& \multirow{3.3}{0.5in}{Number} 
& Generating Function &\vspace{-.05in} $\dfrac{x}{\sqrt{1-4x}}-x$ \\
\cline{3-4}
& & Coefficient & \vspace{0.02in} $\displaystyle {2n-2 \choose n-1}$ \\
\cline{3-4}
& & Asymptotic Coefficient & $4^n \left( \dfrac{1}{4 \sqrt{\pi n}}+\dfrac{3}{32 \sqrt{\pi n^3}} + \dfrac{25}{512 \sqrt{\pi n^5}} \right) +\bigo{\dfrac{4^n}{n^{7/2}}}$ \\
\cline{2-4}
& \multirow{3.3}{0.5in}{Edges}
& Generating Function &\vspace{-.05in} $\dfrac{x}{1-4x}-\dfrac{x}{\sqrt{1-4x}}$ \\
\cline{3-4}
& & Coefficient & \vspace{0.02in} $\displaystyle 4^{n-1}-{2n-2 \choose n-1}$ \\
\cline{3-4}
& & Asymptotic Coefficient & $4^n \left(\dfrac{1}{4}- \dfrac{1}{4 \sqrt{\pi n}}-\dfrac{3}{32 \sqrt{\pi n^3}} - \dfrac{25}{512 \sqrt{\pi n^5}} \right) +\bigo{\dfrac{4^n}{n^{7/2}}}$ \\
\cline{2-4}
& \multirow{2.2}{0.5 in}{Expected Length} 
&Coefficient & $\dfrac{(2n-2)!!}{(2n-3)!!}-1$\\
\cline{3-4}
& & Asymptotic Coefficient& $\sqrt{\pi n} -1- \dfrac{3 \sqrt{\pi}}{8 \sqrt{n}}- \dfrac{7 \sqrt{\pi}}{128 \sqrt{n^3}} + \bigo{\dfrac{1}{n^{5/2}}}$ \\
\hline

\end{tabular}
}
\end{table}

\begin{table}
\makebox[1 \textwidth][c]{ 
\small
\def\arraystretch{2}
\begin{tabular}{|m{.8cm} | c | c | m{8 cm} |}

\hline
\multicolumn{4}{|c|}{Binary Trees, cont.} \\
\hline
\multirow{6}{0.5in}{ \begin{sideways}  \parbox{1.5in}{\begin{center}Downward Paths Which End in Leaves\end{center}}  \hspace{0.55in} \end{sideways}} 
& \multirow{3.3}{0.5in}{Number} 
& Generating Function &\vspace{-.05in} $\displaystyle \dfrac{x}{1-4x}-\dfrac{x}{\sqrt{1-4x}}$ \\
\cline{3-4}
& & Coefficient & \vspace{0.02in} $\displaystyle4^{n-1}-{2n-2 \choose n-1}$ \\
\cline{3-4}
& & Asymptotic Coefficient & $4^n \left(\dfrac{1}{4}- \dfrac{1}{4 \sqrt{\pi n}}-\dfrac{3}{32 \sqrt{\pi n^3}} - \dfrac{25}{512 \sqrt{\pi n^5}} \right) +\bigo{\dfrac{4^n}{n^{7/2}}}$ \\
\cline{2-4}& \multirow{3.3}{0.5in}{Edges}
& Generating Function & \vspace{-.05in}$\dfrac{x}{(1-4x)^{3/2}}-\dfrac{x}{1-4x}$ \\
\cline{3-4}
& & Coefficient & $\dfrac{(2n-1)!}{((n-1)!)^2}-4^{n-1}$ \\
\cline{3-4}
& &Asymptotic Coefficient & $4^n \left(\dfrac{\sqrt{n}}{2\sqrt{\pi}}-\dfrac{1}{4}-\dfrac{1}{16 \sqrt{\pi n}} + \dfrac{1}{256 \sqrt{\pi n^3}} 
\right)+\bigo{\dfrac{4^n}{n^{5/2}}}$ \\
\cline{2-4}
& \multirow{2.2}{0.5 in}{Expected Length} 
&Coefficient & $\dfrac{(2n-1)!!-(2n-2)!!}{(2n-2)!!-(2n-3)!!}$\\
\cline{3-4}
& & Asymptotic Coefficient& $\dfrac{2\sqrt{n}}{\sqrt{\pi}}-1+\dfrac{2}{\pi}+\dfrac{8-5\pi}{4 \pi^{3/2} \sqrt{n}} + \dfrac{4-\pi}{2 \pi^2 n}+\bigo{\dfrac{1}{n^{3/2}}}$ \\
\hline
\multirow{6}{0.5in}{\hspace{.2cm}\begin{sideways}  \parbox{1.8in}{\begin{center} Arbitrary Paths \end{center}}  \hspace{0.3in} \end{sideways}} 
& \multirow{3.3}{0.5in}{Number} 
& Generating Function & \vspace{-.05in}$\dfrac{9}{4(1-4x)^{3/2}}+\dfrac{3}{4\sqrt{1-4x}}+\dfrac{1}{2x}-\dfrac{1}{2x(1-4x)^{3/2}}$ \\
\cline{3-4}
& & Coefficient & \vspace{0.02in}$\displaystyle {n \choose 2} \dfrac{1}{n+1} {2n \choose n}$ \\
\cline{3-4}
& & Asymptotic Coefficient & $4^n \left( \dfrac{\sqrt{n}}{2\sqrt{\pi}} -\dfrac{17}{16\sqrt{\pi n}} + \dfrac{289}{256 \sqrt{\pi n^3}}  \right) + \bigo{\dfrac{4^n}{n^{5/2}}}$ \\ 
\cline{2-4}
& \multirow{3.3}{0.5in}{Edges}
& Generating Function & \vspace{-.05in}$\dfrac{2x}{(1-4x)^2}-\dfrac{1}{(1-4x)^{3/2}}+\dfrac{1}{1-4x}$ \\
\cline{3-4}
& & Coefficient & $\left(\dfrac{n+2}{2} \right)4^n-\dfrac{(2n+1)!}{(n!)^2}$ \\
\cline{3-4}
& & Asymptotic Coefficient &$4^n \left( \dfrac {n}{2} - \dfrac{2 \sqrt{n}}{\sqrt{\pi}} +1 - \dfrac{3}{4\sqrt{ \pi n}} + \dfrac{7}{64 \sqrt{\pi n^3}} \right) + \bigo{\dfrac{4^n}{n^{5/2}}}$ \\ 
\cline{2-4}
& \multirow{1.2}{0.5 in}{Expected Length} 
& Asymptotic Coefficient& $\sqrt{\pi n} -4+\dfrac{33 \sqrt{\pi}}{8 \sqrt{n}} - \dfrac{10}{n} + \dfrac{833\sqrt{\pi}}{128 n^{3/2}}-\dfrac{12}{n^2} + \bigo{\dfrac{1}{n^{5/2}}}$ \\ 
\hline
\multirow{6}{0.5in}{\hspace{.2cm}\begin{sideways}  \parbox{1.8in}{\begin{center} Paths Between Leaves \end{center}}  \hspace{0.265in} \end{sideways}} 
& \multirow{2.2}{0.5in}{Number} 
& Generating Function &\vspace{-.05in} $\dfrac{x^3}{1-4x}$ \\
\cline{3-4}
& & Coefficient & $\displaystyle 4^{n-3}$ \\
\cline{2-4}
& \multirow{3.3}{0.5in}{Edges}
& Generating Function &$\dfrac{2x^3}{(1-4x)^{3/2}}$ \\
\cline{3-4}
& & Coefficient & $\dfrac{2(2n-5)!}{((n-3)!)^2}$ \\
\cline{3-4}
& & Asymptotic Coefficient & $\displaystyle 4^n \left(\frac{\sqrt{n}}{16 \sqrt{\pi}} - \frac{9}{128 \sqrt{\pi n}} - \frac{79}{2048 \sqrt{\pi n^3}}\right) + \bigo{\frac{4^n}{n^{5/2}}}$ \\
\cline{2-4}
& \multirow{2.2}{0.5 in}{Expected Length} 
&Coefficient & $\dfrac{2(2n-5)!!}{(2n-6)!!}$\\
\cline{3-4}
& & Asymptotic Coefficient& $\displaystyle \frac{4\sqrt{n}}{\sqrt{\pi}} - \frac{9}{2 \sqrt{\pi n}} - \frac{79}{32 \sqrt{\pi n^3}} + \bigo{\frac{1}{n^{5/2}}}$ \\
\hline
\end{tabular}
}
\end{table}

\clearpage

\bibliographystyle{amsplain}
\bibliography{bib}

\providecommand{\bysame}{\leavevmode\hbox to3em{\hrulefill}\thinspace}
\providecommand{\MR}{\relax\ifhmode\unskip\space\fi MR }
\providecommand{\MRhref}[2]{%
  \href{http://www.ams.org/mathscinet-getitem?mr=#1}{#2}
}
\providecommand{\href}[2]{#2}
\begin{thebibliography}{1}

\bibitem{cheon}
Gi-Sang Cheon and Louis~W. Shapiro, \emph{Protected points in ordered trees},
  Appl Math Lett \textbf{21} (2008), 516--520.

\bibitem{copenhaver_2017}
Keith Copenhaver, \emph{k-protected vertices in unlabeled rooted plane trees},
  Graphs and Combinatorics \textbf{33} (2017), no.~2, 347–355.

\bibitem{de_bruijn_knuth_rice_1972}
N.G. de~Bruijn, D.E. Knuth, and S.E. Rice, \emph{The average height of planted
  plane trees}, Graph Theory and Computing (1972), 15–22.

\bibitem{entringer_meir_moon_szekely_1994}
R.C. Entringer, A.~Meir, J.W. Moon, and L.A. Sz{\'e}kely, \emph{On the wiener
  index of trees from certain families}, Australasian Journal of Combinatorics
  \textbf{10} (1994), 211--224.

\bibitem{flajolet_odlyzko_1982}
Philippe Flajolet and Andrew Odlyzko, \emph{The average height of binary trees
  and other simple trees}, Journal of Computer and System Sciences \textbf{25}
  (1982), no.~2, 171–213.

\bibitem{flajoletsedgewick}
Philippe Flajolet and Robert Sedgewick, \emph{Analytic combinatorics},
  Cambridge University Press, 2009.

\bibitem{heuberger_prodinger_2017}
Clemens Heuberger and Helmut Prodinger, \emph{Protection number in plane
  trees}, Applicable Analysis and Discrete Mathematics \textbf{11} (2017),
  no.~2, 314–326.

\end{thebibliography}
\end{document}